\documentclass[12pt]{amsart}
\usepackage[applemac]{inputenc}

\usepackage[colorlinks=true,linkcolor=magenta,citecolor=blue,urlcolor=blue,hypertexnames=false,linktocpage]{hyperref}
\usepackage{bookmark}
\usepackage{amsthm,thmtools,amssymb,amsmath,amscd,amsfonts}

\usepackage{fancyhdr}
\usepackage{esint}
\usepackage{enumerate}

\usepackage{pictexwd,dcpic}

\declaretheorem[name=Theorem,numberwithin=section]{thm}
\declaretheorem[name=Remark,style=remark,sibling=thm]{rem}
\declaretheorem[name=Lemma,sibling=thm]{lemma}
\declaretheorem[name=Proposition,sibling=thm]{prop}

\declaretheorem[name=Corollary,sibling=thm]{cor}

\numberwithin{equation}{section}

\usepackage{cleveref}
\crefname{lemma}{Lemma}{Lemmata}
\crefname{prop}{Proposition}{Propositions}
\crefname{thm}{Theorem}{Theorems}
\crefname{cor}{Corollary}{Corollaries}
\crefname{defn}{Definition}{Definitions}
\crefname{example}{Example}{Examples}
\crefname{rem}{Remark}{Remarks}
\crefname{assum}{Assumption}{Assumptions}
\crefname{nota}{Notation}{Notation}
\usepackage{autonum}

\newcommand{\ti}{\tilde}

\newcommand{\bs}{\backslash}
\newcommand{\cn}{\colon}
\newcommand{\sub}{\subset}

\newcommand{\bbR}{\mathbb{R}}

\newcommand{\bbS}{\mathbb{S}}

\newcommand{\8}{\infty}

\newcommand{\al}{\alpha}
\newcommand{\be}{\beta}
\newcommand{\ga}{\gamma}
\newcommand{\de}{\delta}
\newcommand{\ep}{\varepsilon}
\newcommand{\ka}{\kappa}
\newcommand{\la}{\lambda}
\newcommand{\om}{\omega}
\newcommand{\si}{\sigma}

\newcommand{\ph}{\varphi}

\newcommand{\Ga}{\Gamma}

\newcommand{\La}{\Lambda}

\newcommand{\cL}{\mathcal{L}}

\newcommand{\cF}{\mathcal{F}}

\newcommand{\del}{\partial}
\newcommand{\n}{\nabla}

\newcommand{\fa}{\forall}
\newcommand{\rt}{\sqrt}

\newcommand{\ip}[2]{\left\langle #1,#2 \right\rangle}
\newcommand{\fr}[2]{\frac{#1}{#2}}
\newcommand{\x}{\times}

\DeclareMathOperator{\const}{const}

\newcommand{\pf}[1]{\begin{proof}#1 \end{proof}}
\newcommand{\eq}[1]{\begin{equation}\begin{alignedat}{2} #1 \end{alignedat}\end{equation}}

\newcommand{\br}[1]{\left(#1\right)}
\newcommand{\abs}[1]{\lvert #1\rvert}
\newcommand{\enum}[1]{\begin{enumerate}[(i)] #1 \end{enumerate}}
\newcommand{\enu}[1]{\begin{enumerate}[(a)] #1 \end{enumerate}}

\newcommand{\ra}{\rightarrow}
\newcommand{\hra}{\hookrightarrow}

\newcommand{\mrm}{\mathrm}

\newcommand{\q}{\quad}

\begin{document}

\title[]{Christoffel-Minkowski flows}
\author[P. Bryan, M. N. Ivaki, J. Scheuer]{Paul Bryan, Mohammad N. Ivaki, Julian Scheuer}
%\author[P. Bryan]{Paul Bryan}
%\address{Department of Mathematics, Macquarie University
%NSW 2109, Australia}
%\email{\href{paul.bryan@mq.edu.au}{paul.bryan@mq.edu.au}}
%\author[M.N. Ivaki]{Mohammad N. Ivaki}
%\address{Department of Mathematics, University of Toronto, Ontario,
%M5S 2E4, Canada}
%\email{\href{m.ivaki@utoronto.ca}{m.ivaki@utoronto.ca}}
%\author[J. Scheuer]{Julian Scheuer}
%\address{Department of Mathematics, Columbia University
%New York, NY 10027, USA}
%\email{\href{mailto:jss2291@columbia.edu}{jss2291@columbia.edu}}
\dedicatory{}
\subjclass[2010]{}
\keywords{Christoffel--Minkowski problem; Curvature flow}
\date{\today}
\begin{abstract}
We provide a curvature flow approach to the regular Christoffel--Minkowski problem. The speed of our curvature flow is of an entropy preserving type and contains a global term.
\end{abstract}
\maketitle
\tableofcontents
\section{Introduction}

The regular Christoffel--Minkowski problem asks for sufficient and necessary conditions on a positive function $\varphi\in C^{\8}(\bbS^{n})$ such that there exists a smooth strictly convex body $K$ with support function $s$ satisfying
\eq{\label{eq:CMP}p_k(\bar \n^{2}s+s\bar g )=\varphi.}
Here $\bar g$ is the round metric on $\bbS^{n}$, $\bar \n$ its Levi-Civita connection and $p_{k}$ is the $k$-th elementary symmetric polynomial of the eigenvalues of its argument, which in the case of \eqref{eq:CMP} are the principal radii of curvature for the boundary of $K,$ $\del K$. Below we will introduce this problem in further detail. In \cite[Thm.~1.3]{Guan2003a}, this problem was solved with the help of a sophisticated constant rank theorem under the following hypotheses:

\enum{
\item \label{invariance}\eq{\int_{\bbS^{n}}u\varphi(u)d\om=0,}
\item \label{weak-convex}\eq{\bar \n^{2}\varphi^{-\frac{1}{k}}+\varphi^{-\frac{1}{k}}\bar g\geq  0,}
\item \label{homotopy} There is a continuous (with respect to the $C^{\8}$-norm) homotopy from $\ph$ to the constant one function, which respects \eqref{weak-convex}.
}
Here $d\om$ is the standard volume form of $\mathbb{S}^n$.
The first of these assumptions is necessary due to the translation invariance of the problem. Assumption \eqref{weak-convex} is sufficient but not necessary; see below. In \cite{Guan2003a} a flow approach by Andrews and the authors was announced for removing \eqref{homotopy}; however, the work has never appeared. A direct verification for the validity of \eqref{homotopy} was presented in \cite{Sheng2004}.

Since the elliptic method of \cite{Guan2003a} was introduced, it has remained an intriguing question whether these problems can be solved with a curvature flow as well. To our knowledge, except in case of the Minkowski problem where $k=n$, (cf., \cite{Chou2000}), such an approach is missing. The purpose of this paper is to fill this gap and  to reprove the results from \cite{Guan2003a, Sheng2004} with the help of a carefully designed curvature flow.

Our proof has two main advantages compared to the existing ones: First, in the case that \eqref{weak-convex} holds with strict inequality, we do not make use of the constant rank theorem. Second, a byproduct of our proof is an improvement and a simplification of how \eqref{homotopy} can be removed.

The idea of this paper is inspired by the literature on volume preserving curvature flows. Under the strong convexity assumption, i.e., \eqref{weak-convex} holds with the strict sign, we consider a curvature flow whose fixed points are the solutions to the Christoffel--Minkowski problem. The full result then follows from a simple approximation argument. The main result of this paper is as follows:

\begin{thm}[Christoffel--Minkowski flow]\label{CM-flow}
Suppose $1\leq k\leq n$ and $0<\ph\in C^{\8}(\bbS^{n})$ satisfies
\enu{
\item\eq{\label{integral} \int_{\mathbb{S}^n}u\varphi(u)d\omega=0,}
\item \eq{\label{strict-convex}\bar \n^{2}\varphi^{-\frac{1}{k}}+\varphi^{-\frac{1}{k}}\bar g>  0.}
}
Then for an arbitrary strictly convex initial hypersurface $M_{0}=x(0,\bbS^{n})$, there exists a unique family of embeddings
\eq{x\cn [0,\8)\x\bbS^{n}\ra \bbR^{n+1},}
which satisfies the constrained curvature flow
\eq{\label{CM-flow-2}\dot{x}=\br{\fr{\int_{\bbS^{n}}p_{k}^{\fr{k-1}{k}}d\om}{\int_{\bbS^{n}}\ph^{-\fr 1k}p_{k}d\om}\ph^{-\fr 1k}(\nu)-p_{k}^{-\fr 1k}(\la)}\nu.}
Here $\la=(\la_{i})$ are the principal radii of curvature and $\nu$ is the outward pointing normal. If $C(K_{t})$ denotes the centroid of the convex body bounded by
\eq{M_{t}=x(t,\bbS^{n})=\del K_{t},}
then the embeddings
\eq{\ti x(t,\cdot)=x(t,\cdot)-C(K_{t})}
converge smoothly to a solution of the Christoffel--Minkowski problem.
\end{thm}

In \Cref{completion} we will weaken the strict convexity assumption on $\ph$ and, in conjunction with \Cref{CM-flow}, reprove the following result.

\begin{cor}[\cite{Guan2003a, Sheng2004}]\label{CMP}
Let $1\leq k< n $ and $0<\varphi\in C^{\infty}(\mathbb{S}^n)$ satisfy
\eq{\bar \n^{2}\varphi^{-\frac{1}{k}}+\varphi^{-\frac{1}{k}}\bar g\geq  0,\quad \int_{\mathbb{S}^n}u\varphi(u)d\omega=0.}
Then there exists a smooth, strictly convex solution to
\eq{p_k(\bar \n^{2}s+s\bar g )=\varphi.}
\end{cor}

Before we proceed to the proofs of these results, we review the origin of the Christoffel--Minkowski problem.

\subsection*{Christoffel--Minkowski problem}
To briefly review the origin of the Christoffel--Minkowski problem in convex geometry, we will recall a few definitions. As always, an excellent reference for this material is \cite{Schneider2013a}. Let us denote the unit sphere and unit ball respectively by $\mathbb{S}^n$ and $B.$ By a convex body, we mean a compact convex set with non-empty interior. The set of convex bodies is denoted by $\mathcal{K}.$ Write $V(K)$ for the volume of $K\in\mathcal{K}$. The support function of $K$ is defined by
\eq{s_K(u)=\max_{x\in K} \langle x,u\rangle\quad \forall u\in \mathbb{S}^n.}
The sum of $K,L\in \mathcal{K}$ is defined as
\[K+L=\{a+b: a\in K,~b\in L\}.\]
For small $\varepsilon>0,$ due to the classical Steiner formula,
\[V(K+\varepsilon B)=\sum_{j=0}^{n+1}\binom{n+1}{j}W_j(K)\varepsilon^j,\]
where $\{W_j(K)\}_j$ are the quermassintegrals of $K$ or the $(n+1-j)$-th surface area. Note that
\eq{W_0(K)=V(K),}
and surface area of $K$ is
 \eq{(n+1)W_1(K)=S(K),} and also
 \eq{W_{n+1}(K)=W_{n+1}(B)}
 is a dimensional constant.

A remarkable theorem states that given any $K\in\mathcal{K}$, there exists a Borel measure $S_i(K)$ on the unit sphere such that for all $L\in \mathcal{K}$,
\eq{
    W_{i}(K,L)&:=\frac{1}{n+1-i}\lim_{\varepsilon\to 0}\frac{W_i(K+\varepsilon L)-W_i(K)}{\varepsilon}\\
    &=\frac{1}{n+1}\int_{\mathbb{S}^n}{s_L(u)}dS_{i}(K,u).}
Due to translation invariance of $W_{i}$,
\eq{
\int_{\mathbb{S}^n}udS_{i}(K,u)=0.
}
Moreover, by the mixed quermassintegrals inequality, we have
\eq{W_i(K,L)^{n+1-i}\geq W_i(K)^{n-i}W_i(L),}
and equality holds if and only if $K$ and $L$ are homothetic. In particular, this implies that
\eq{\label{variational structure}
\min_{L\in \mathcal{K}} \frac{W_i(K,L)}{(W_i(L))^{\frac{1}{n+1-i}}}
}
is attained only for homothetic transformations of $K$.

Suppose $0\leq i<n.$ The Christoffel--Minkowski problem aims at reconstructing the convex body from its $(n+1-i)$-th surface area measure:

Given a Borel measure $\mu$ on the unit sphere with
\eq{\label{1st necessary condition}\int_{\mathbb{S}^n} ud\mu(u)=0,}
find necessary and sufficient conditions on $\mu$ such that there exists a convex body $K$ with $S_i(K)=\mu$.

From \eqref{variational structure}, one would expect that solutions might be found by mini\-mizing
\eq{\mathcal{E}_i(L):=\frac{\int_{\mathbb{S}^n}s_Ld\mu}{(W_i(L))^{\frac{1}{n+1-i}}}}
in a suitable class of convex bodies. In particular, the search may be restricted to the set
\eq{\mathcal{K}'=\{L\in \mathcal{K}: W_i(L)=1,~\operatorname{St}(L)=0,~ \mathcal{E}_i(L)\leq  \mathcal{E}_i(B)\}.
}
where $\operatorname{St}(L)$ denotes the Steiner point of $L.$
In fact, this is the case for $i=0,$ corresponding to the prescribed surface area measure, a.k.a the Minkowski problem, provided
\eq{\label{2nd necessary condition}\mu(\{u:\langle u,v\rangle>0\})>0\quad\forall v\in\mathbb{S}^n.}
Moreover, (\ref{1st necessary condition}) and (\ref{2nd necessary condition}) together are both necessary and sufficient.

However, for $i>0$, this method has not been successful so far. %a necessary and sufficient condition on $\mu$ that forces a uniform upper bound and a uniform positive lower bound on the support functions under the restriction $W_i(\cdot)=1$ and $\operatorname{St}(\cdot)=0$ is not known yet.
For $i=n-1$, i.e., the Christoffel problem, Firey \cite{Firey1967b} and Berg used methods based on a Green function and
subharmonic functions to solve the problem completely; see also \cite{Firey1970} and \cite{Li2019f}.

If the boundary of $K,$ $\partial K,$ is $C^2$ smooth and strictly convex, then
\eq{dS_i(K)=p_{n-i}d\om,}
where
\eq{p_{\ell}:=\sum\limits_{1\leq i_1<\cdots<i_{\ell}\leq n}\lambda_{i_1}\cdots\lambda_{i_{\ell}},}
and  $\{\lambda_i\}$ are the principal radii of curvature (considered as functions on the unit sphere) of $\partial K$.
This brings us to the regular Christoffel--Minkowski problem which asks:

Given a smooth function $\varphi\cn\mathbb{S}^n\to (0,\infty),$ find necessary and sufficient conditions on $\varphi$ such that there exists a smooth, strictly convex hypersurface whose $p_k$ equals $\varphi.$

Guan and Ma in \cite{Guan2003a} using continuity methods and a sophisticated constant rank theorem found a sufficient condition on $\varphi\cn$
\eq{
\bar{\nabla}^2\varphi^{-\frac{1}{k}}+\bar{g}\varphi^{-\frac{1}{k}}\geq 0,\quad \int_{\mathbb{S}^n}u\varphi(u)d\omega=0.
}

Guan--Ma's sufficient condition is not necessary. In fact, \cite[Thm. 6.8]{Zhang1994} states the following. Let $C^{\alpha}_e(\mathbb{S}^n)$ denote the set of $\alpha$-H\"{o}lder continuous, antipodal symmetric functions on $\mathbb{S}^n$, and $\mathcal{F}^{2,\alpha}_e$ denotes the set of origin-symmetric convex bodies with strictly convex, $C^{2,\alpha}$-smooth boundaries. Suppose $\varphi$ is the $p_k$ of a convex body in $\mathcal{F}^{2,\alpha}_e$. Then there exists a $C_e^{\alpha}(\mathbb{S}^n)$ neighborhood $\mathcal{N}$ of $\varphi$ such that every function in $\mathcal{N}$ is the $p_k$ of a convex body in $\mathcal{F}^{2,\alpha}_e$. See also \cite[Cor. 6.9]{Zhang1994}.

\subsection*{Constrained curvature flows}
In order to get the estimates required in the proof of \Cref{CM-flow}, it is convenient to reformulate the problem in terms of the principal curvatures instead of the radii:
\eq{\label{gen-flow}\del_{t}x=(\mu(t)f(\nu)-F(\ka))\nu.}
Here the global term $\mu$ is defined in \Cref{CM-flow}, and
\eq{F=\br{\fr{\si_{n}}{\si_{n-k}}}^{\fr 1k},}
where $\sigma_{k}$ is the $k$-th elementary symmetric polynomial of the principal curvatures $\ka=(\ka_{i})$ and
\eq{f=\ph^{-\fr{1}{k}}.}

Due to their applications, flows of the form \eqref{gen-flow} have been receiving significant attention. In case $f=1$, they appeared in the form of volume-, surface area or quermassintegral preserving flows with global term, e.g.,
\cite{Andrews2001, Andrews2018, Bertini2018, Cabezas-Rivas2010a, Huisken1987a, Ivaki2013a, McCoy2003, McCoy2005, Sinestrari2015a}
and have been used to deduce geometric inequalities for hypersurfaces. In case $\mu=1$ , they have been used to prove problems of prescribed curvature, e.g., \cite{Chou2000, Gerhardt:02/2006}; see also \cite{BIS6} for a much broader overview over this topic.
Except for cases where the flow arises as a rescaling of a purely expanding or contracting flow, our flow \eqref{gen-flow} seems to be the first one with a mixed constraining term, where ``mixed" means that we have a non-local and a (possibly anisotropic) local term.

\subsection*{Outline}
The paper is organized as follows. In \Cref{prelim} we collect some basics, notation and evolution equations. \Cref{pinching} provides the crucial pinching estimate along \eqref{gen-flow} which holds in a much more general setting, as we will point out. In \Cref{monotone}, using Andrews' generalized H\"{o}lder inequality, we prove the monotonicity of certain functionals along the flow and prove uniform estimates for the support function of the adjusted flow hypersurfaces. In \Cref{Curv-est} we obtain the curvature estimates. The higher order estimates in \Cref{sec:LTE} are somewhat non-standard due to the global term. We take some care here and complete the proof of \Cref{CM-flow}. In \Cref{completion} we complete the argument on how to weaken the strict convexity assumption. This section also provides an alternative to the proof of \cite{Sheng2004} that removes \eqref{homotopy}.

\section{Basic conventions and evolution equations}\label{prelim}
We collect briefly our conventions on hypersurfaces of the (flat) Euclidean space $\bbR^{n+1}$. For an embedding
\eq{x\cn M\hra \bbR^{n+1}}
of a closed hypersurface with exterior normal vector field $\nu,$ the second fundamental form is defined so that it is positive definite on spheres:
\eq{D_Y X=\n_Y X-h(X,Y)\nu.}
Here $D$ denotes the connection on $\bbR^{n+1}$ and $\n$ the Levi-Civita connection on $M$ with respect to the induced metric
\eq{g=x^{\ast}\ip{\cdot}{\cdot}.}
In a local coordinate frame $(e_i)_{1\leq i\leq n}$ for $M$ and with the convention $\n_i=\n_{e_i},$
the Codazzi equation implies that
\eq{\n_i \n_j e_k-\n_j\n_i e_k=h_{jk}h^m_i \del_m x-h_{ik}h^m_j \del_m x,
}
where the index was lifted using $g$.
Hence the covariant Riemann tensor is given by
\eq{R_{ijkl}:={R_{ijk}}^m g_{lm}=(\n_i \n_j e_k-\n_j\n_i e_k)^m g_{lm}=h_{il}h_{jk}-h_{ik}h_{jl}.
}
For brevity, we introduce semi-colons to denote indices of covariant derivatives. For example, if $T$ is a tensor, the components of its second derivative are denoted by
\eq{T_{;ij}=\n_j\n_i T - \n_{\n_j\del_i}T.
}
The support function of a convex body $K$ is defined by
\eq{s_{K}(u)=\max_{x\in K}\ip{x}{u},\quad u\in \bbS^{n}.}
For a (strictly) convex body $K$, the bilinear form
\eq{r:=\bar\n^{2}s_{K}+s_{K}\bar g }
is (positive) non-negative definite. The eigenvalues of $r$ with respect to $\bar g$ are the principal radii of curvature.

We will use common facts about curvature functions, i.e., functions that can either be viewed as symmetric functions of the eigenvalues of the Weingarten map or as depending on the second fundamental form and the metric
\eq{F=F(\ka_{i})=F(A)=F(g,h).}
Then we write
\eq{F^{ij}=\fr{\del F}{\del h_{ij}},\quad F^{ij,kl}=\fr{\del^2 F}{\del h_{ij}\del h_{kl}}.}
We refer to \cite{Andrews:/2007} for a detailed account and common properties.

The short time existence for \eqref{gen-flow} follows from standard arguments, see \cite{Makowski:01/2013} for the detailed procedure.
We derive the evolution equations for the flow \eqref{gen-flow} on a time interval $[0,T)$, $T<\infty$. We use the linearized operator
\eq{\cL =\del_{t}-F^{kl}\nabla_k\nabla_l.}
\begin{lemma}
The following evolution equations along the flow \eqref{gen-flow}.
\eq{\del_t g_{ij}=2(\mu(t)f-F) h_{ij}.}
\eq{\label{Ev-h}\mathcal{L}h_i^j=&F^{kl}h_{km}h^{m}_{l}h^{j}_{i}-\mu f h^{j}_{m}h^{m}_{i}+F^{kl,rs}h_{kl;i}{h_{rs;}}^{j}\\
			&-\mu \bar\n^{2} f(x_{;k},x_{;l})h^{k}_{i}h^{lj}-\mu \bar \n f(x_{;k})h^{kj}_{i;}.}
\eq{\label{Ev-F}\cL F=&F^{ij}h_{ik}h^{k}_{j}F-\mu F^{ij}h_{ik}h^{k}_{j}f \\
&-\mu F^{ij}\bar\n^2 f(x_{;k},x_{;l})h^{k}_{i}h^{l}_{j}-\mu {F_{;}}^{k}\bar\n f(x_{;k}).}
For every fixed $x_{0}\in \bbR^{n+1}\cn$
\eq{\label{Ev-s}\cL\ip{x-x_{0}}{\nu}=&F^{ij}h_{ik}h^{k}_{j}\ip{x-x_{0}}{\nu}-2F\\
            &+\mu(f-\ip{x-x_{0}}{\n f}).}

\end{lemma}

\pf{
The evolution of the metric is standard. We start by calculating the evolution of the Weingarten operator. It follows from \cite[Lem.~2.3.3]{Gerhardt:/2006} that
\eq{\label{Ev-A-gen-b}\dot{h}^{j}_{i}&=-{(\mu f-F)_{;i}}^{j}-(\mu f-F)h_{ik}h^{kj}.}
First we have to replace the term ${F_{;i}}^{j}$. There holds
\eq{{F_{;i}}^{j}={F^{kl}h_{kl;i}}^{j}+F^{kl,rs}h_{kl;i}{h_{rs;}}^{j}.}
Now we use the Codazzi, Weingarten and Gauss equation to deduce
\eq{h_{kl;ij}=&h_{ki;lj}\\
		=&h_{ki;jl}+{R_{ljk}}^{m}h_{mi}+{R_{lji}}^{m}h_{mk}\\
		=&h_{ij;kl}+(h_{jk}h^{m}_{l}h_{mi}-h_{lk}h^{m}_{j}h_{mi})+(h_{ji}h^{m}_{l}h_{mk}-h_{li}h^{m}_{j}h_{mk}).}
%		=&h_{ij;kl}+(h_{jk}h^{m}_{l}h_{mi}-h_{lk}h^{m}_{j}h_{mi})+(h_{ji}h^{m}_{l}h_{mk}-h_{li}h^{m}_{j}h_{mk})\\
%		=&h_{ij;kl}+(h_{jk}h^{m}_{l}h_{mi}-h_{lk}h^{m}_{j}h_{mi})+(h_{ji}h^{m}_{l}h_{mk}-h_{li}h^{m}_{j}h_{mk}).}
We use
\eq{F^{k}_{l}h^{l}_{m}=h^{k}_{l}F^{l}_{m}}
to get
\eq{F^{kl}h_{kl;ij}=&F^{kl}h_{ij;kl}-F^{kl}h_{lk}h^{m}_{j}h_{mi}+F^{kl}h_{ij}h^{m}_{l}h_{mk}.}
From the $1$-homogeneity of $F$ it follows that
\eq{F^{kl}h_{kl;ij}=&F^{kl}h_{ij;kl}-Fh^{m}_{j}h_{mi}+F^{kl}h^{m}_{l}h_{mk}h_{ij}.}
Inserting this into \eqref{Ev-A-gen-b} we obtain
\eq{\label{Ev-h-1}
\cL h^{j}_{i}=&-F^{kl}h^{j}_{i;kl}+{F^{kl}h_{kl;i}}^{j}+F^{kl,rs}h_{kl;i}{h_{rs;}}^{j}\\
			&-{{\mu f_{;i}}^{j}}-(\mu f-F)h_{ik}h^{kj}\\
			=&F^{kl}h^{m}_{l}h_{mk}h^{j}_{i}-\mu fh_{i}^{m}h^{j}_{m}+F^{kl,rs}h_{kl;i}{h_{rs;}}^{j}-\mu{f_{;i}}^{j}.}
We calculate the term ${f_{;i}}^{j}$. We may assume that $f$ is extended as a zero homogeneous function to $\bbR^{n+1}$. Due to the Gaussian formula,
\eq{x_{;kj}=- h_{kj}\nu,}
we obtain
\eq{{f_{;i}}=f_{\nu}(\nu_{;i}),\q f_{;ij}=f_{\nu\nu}(\nu_{;i},\nu_{;j})+f_{\nu}(\nu_{;ij}).}
Moreover, by the Weingarten equation we have
\eq{\nu_{;i}=h^{k}_{i}x_{;k},\q \nu_{;ij}=h^{k}_{i;j}x_{;k}- h^{k}_{i}h_{kj}\nu.}
 The result follows from inserting the expression for $f_{;ij}$ into \eqref{Ev-h-1} and using the zero homogeneity in $\nu$.

 For $F$ we have
\eq{\del_{t}F=F^{i}_{j}\del_{t}h^{j}_{i}=F^{ij}F_{;ij}-\mu F^{ij}f_{;ij}+F^{ij}h_{ik}h^{k}_{j}(F-\mu f) }
and hence
\eq{\cL F=&F^{ij}h_{ik}h^{k}_{j}F-\mu F^{ij}h_{ik}h^{k}_{j}f \\
&-\mu F^{ij}\bar \n^2 f(x_{;k},x_{;l})h^{k}_{i}h^{l}_{j}-\mu {F_{;}}^{k}\bar\n f(x_{;k}).}
For the support function we calculate
\eq{\partial_{t}\ip{x-x_{0}}{\nu}=&\mu f-F-\ip{x-x_{0}}{\n(\mu f-F)},\\
\del_{i}\ip{x-x_{0}}{\nu}=&h^{k}_{i}\ip{x-x_{0}}{x_{;k}}}
and
\eq{\ip{x-x_{0}}{\nu}_{;ij}=h^{k}_{i;j}\ip{x-x_{0}}{x_{;k}}+h_{ij}-h_{ik}h^{k}_{j}\ip{x-x_{0}}{\nu}.}
Therefore,
\eq{\cL \ip{x-x_{0}}{\nu}=\mu f-2F-\mu\ip{x-x_{0}}{\n f}+F^{ij}h_{ik}h^{k}_{j}\ip{x-x_{0}}{\nu}.}

}

\section{Pinching estimate}\label{pinching}

\begin{prop}\label{pinching est}
There exist $\ep_0$, depending on $M_0$, and $f$, such that along the flow \eqref{gen-flow} with initial hypersurface $M_0$ we have
\eq{h_{ij}\geq \ep_0 H g_{ij}.}
\end{prop}

\pf{
Define
\eq{S_{ij}=h_{ij}-\ep Hg_{ij}.}
For the moment, we choose $\ep_0$ sufficiently small, such that $(S_{ij})$ is positive definite at $t=0$, for all $\varepsilon\leq \varepsilon_0.$
We calculate
 \eq{\cL S_{ij}=&\cL h_{ij}-\ep\cL H g_{ij}-\ep H\del_t{g}_{ij}\\
 			=&F^{kl}h_{km}h^{m}_{l}S_{ij}-(2F-\mu f)h_{mj}h^{m}_{i}+F^{pq,rs}h_{pq;i}{h_{rs;j}}\\
			&-\mu \bar\n^{2}f(x_{;m},x_{;l})h^{m}_{i}h^{l}_{j}-\mu \bar\n f(x_{;m}){S_{ij;}}^{m}+\ep\mu f\abs{A}^{2}g_{ij}\\
			&-\ep F^{kl,rs}h_{kl;m}{h_{rs;}}^{m}g_{ij}+\ep\mu \bar\n^{2}f(x_{;k},x_{;l})h^{k}_{m}h^{lm}g_{ij}\\
				&-2\ep H(\mu f-F)h_{ij}\\
			=& N_{ij}+\ti N_{ij},}
where we define

\eq{\ti N_{ij}=F^{pq,rs}h_{pq;i}{h_{rs;j}}-\ep F^{kl,rs}h_{kl;m}{h_{rs;}}^{m}g_{ij}. }
According to Andrews' tensor maximum principle \cite[Thm.~3.2]{Andrews:/2007}, in order to prove that $(S_{ij})$ remains positive definite if so initially, we have to show
\eq{\label{Pinching-1}N_{ij}v^{i}v^{j}+\ti N_{ij}v^{i}v^{j}+2\sup_{\Ga}F^{kl}(2\Ga^{p}_{k}S_{ip;l}v^{i}-\Ga^{p}_{k}\Ga^{q}_{l}S_{pq})\geq 0,}
whenever $(S_{ij})$ is non-negative definite and $S_{ij}v^{i}=0$.

The curvature function $F$ has all the properties which are needed to apply \cite[Thm.~4.1]{Andrews:/2007}. This proves that the sum of the second and third term of \eqref{Pinching-1} is non-negative. Hence we only have to deal with the zero order terms. Suppose $v$ has unit length. Using
\eq{h_{ij}v^{i}=\ep H v_{j},}
we calculate

\eq{N_{ij}v^{i}v^{j}=&-\ep^{2}H^{2}(2F-\mu f)+\ep\mu f\abs{A}^{2}-2\ep^{2}H^{2}(\mu f-F)\\
				&-\mu \ep^{2}H^{2} \bar\n^{2}f(x_{;m},x_{;l})v^{m}v^{l}+\mu\ep\bar\n^{2}f(x_{;k},x_{;l})h^{k}_{m}h^{ml}\\
				=&\ep \mu f(\abs{A}^{2}-\ep H^{2})-\mu \ep^{2}H^{2} \bar\n^{2} f(x_{;m},x_{;l})v^{m}v^{l}\\
				&+\mu\ep \bar\n^{2}f(x_{;k},x_{;l})h^{k}_{m}h^{ml}.}

Due to \eqref{strict-convex} and compactness, there exist constants $\al,\be>0$, such that for all $w\in T\bbS^{n}$,
\eq{\be\abs{w}^{2}\geq\bar \n^{2}f(w,w)\geq (\al-1)f\abs{w}^{2}.}
Hence, using $T_{x}M=T_{\nu(x)}\bbS^{n}$, we get
\eq{N_{ij}v^{i}v^{j}\geq &\ep \mu f(\abs{A}^{2}-\ep H^{2})-\mu \be \ep^{2}H^{2}+\ep\mu (\al-1)f\abs{A}^{2}\\
				\geq &\ep\mu f\br{\al\abs{A}^{2}-\ep H^{2}-\tfrac{\be\ep}{f}H^{2}}\\
				>&0,}
provided $\ep\leq \ep_{0}$, where $\ep_{0}$ is a constant solely depending on $f$ and $M_0$, chosen small enough to ensure $S_{ij}|_{M_0}>0$ and such that
\eq{\ep\br{1+\fr{\be}{f}}<\frac{\alpha}{n}.}
}

\begin{rem}
Note that in the proof of \Cref{pinching est} we have not used the special structure of either the global term or the curvature function $F$. Under the strict convexity assumption of $f$ the proposition is valid for every global term and every curvature function $F$ to which Andrews' maximum principle applies. The class of such $F$ is quite large; see \cite{Andrews:/2007}. This observation might be valuable in similar prescribed curvature problems.
\end{rem}

\begin{rem}
The pinching estimate, \Cref{pinching est}, ensures convexity is preserved since $S_{ij}$ remains positive definite if initially so. In fact, for our arguments, it is not a priori necessary to assume convexity is preserved. For if we let $T^{\ast} \leq T$ be the maximal time for which the flow is convex, then all our estimates are valid on $[0, T^{\ast})$. In particular, \Cref{lower bound on F} gives a positive lower bound on $F$ and the flow remains strictly convex on $[0, T^{\ast})$ yielding a contradiction if $T \ne T^{\ast}$.
\end{rem}

\section{Monotone quantity}\label{monotone}
In the following we suppress any subscripts of the support function and understand it to depend on time,
\eq{s=s_{K_{t}}.}
\begin{lemma}\label{Monotone Q}
Along the flow \eqref{CM-flow-2} we have
\eq{\label{k+1-volume}\frac{d}{dt}\int_{\mathbb{S}^n} sp_kd\omega=0}
and
\eq{\label{weighted L1 mean}\frac{d}{dt}\int_{\mathbb{S}^n}s\ph d\omega\leq 0.}
Moreover, equality holds if and only if $p_{k}\ph^{-1}$ is constant.
\end{lemma}
\begin{proof}
Using
\eq{\dot{s}=\mu \ph^{-\fr 1k}-p_k^{-\frac{1}{k}}}
and that $\bar{\nabla}_ip_k^{ij}=0$ (cf., \cite[Lem.~2-12]{Andrews1994}), we find
\eq{\frac{d}{dt}\int_{\mathbb{S}^n} sp_kd\omega=0.}
Now we calculate
\eq{\label{derivative of functional}\frac{d}{dt}\int_{\mathbb{S}^n}s\ph d\omega
=&\fr{\int_{\bbS^{n}}p_{k}^{\fr{k-1}{k}}d\om}{\int_{\bbS^{n}}\ph^{-\fr 1k} p_{k}d\om}\int_{\bbS^{n}}\ph^{\fr{k-1}{k}}d\omega-\int_{\mathbb{S}^n}\ph p_{k}^{-\fr 1k}d\omega.}
Define
\eq{d\theta=\ph^{\fr{k-1}{k}}d\omega,\q G(x)=x^{-k},\q \zeta=\ph^{\fr 1k}p_{k}^{-\fr 1k}.}
Due to Andrews' generalized H\"{o}lder inequality \cite[Lem.~I3.3]{Andrews1998},
%\eq{\frac{d}{dt}\int_{\mathbb{S}^n}\frac{s}{f^k}d\omega=\int_{\bbS^n}d\mu-\frac{\int_{\bbS^n} G(\zeta)d\mu}{\int_{\bbS^n}\zeta G(\zeta)d\mu}\int_{\bbS^n}\zeta d\mu\leq 0.}
\eq{\fr{d}{dt}\int_{\bbS^{n}}s\ph d\om=\fr{\int_{\bbS^{n}}\zeta G(\zeta)d\theta}{\int_{\bbS^{n}}G(\zeta)d\theta}\int_{\bbS^{n}}d\theta-\int_{\bbS^{n}}\zeta d\theta\leq 0
}
and equality holds if and only if $\zeta$ is constant.
\end{proof}

We write $w_-$ and $w_+$ respectively for the minimum width and the maximum width of a closed, convex hypersurface (or a convex body) with support function $s$. They are defined as
\eq{w_+=\max_{u\in\bbS^n}(s(u)+s(-u)),\quad w_-=\min_{u\in\bbS^n}(s(u)+s(-u)).}

\begin{lemma}\label{lower-upper bnds on sp func}
Let $C(K_{t})$ denote the centroid of $K_{t}$. Then along the flow \eqref{CM-flow-2} the support functions $\ti s$ of
\eq{\ti M_t:=M_t-C(K_{t})} are uniformly bounded above and below away from zero. In particular, the in-radii of the $M_t$ are uniformly bounded below away from zero.
\end{lemma}
\pf{Note that
\eq{\ti s(u)=s(u)-\langle u, C(K_{t})\rangle\q \forall u\in \bbS^n.}
Therefore, due to \eqref{weighted L1 mean} and \eqref{integral}, we have
\eq{\int_{\bbS^n}\ti s\ph d\omega&=\int_{\bbS^n}s\ph d\omega-\left\langle C(K_{t}),\int_{\bbS^n}u\ph(u)d\omega\right\rangle\\
&=\int_{\bbS^n}s\ph d\omega\leq\int_{\bbS^n}s_{K_{0}}\ph d\omega.}
Suppose $\max \ti s$ is attained at the north pole $e_{n+1}$. By convexity
\eq{\ti s(x)\geq \ti s(e_{n+1})x_{n+1}\q\mbox{on}~ \{x_{n+1}>0\}\cap \mathbb{S}^n.} Therefore,
\eq{
\ti s(e_{n+1})\fr{\min\ph}{2}\int_{\{x_{n+1}>\frac{1}{2}\}\cap \mathbb{S}^n}d\omega\leq \int_{\bbS^n}s_{K_{0}}\ph d\omega.
}
Thus $0<\ti s\leq a$ for some constant $a$. By \eqref{k+1-volume} and that the mixed volume is monotonic in each argument (cf., \cite[(5.25)]{Schneider2013a}), we obtain
\eq{
b^{k+1}:=\frac{1}{|\bbS^n|}\int_{\bbS^n}sp_kd\omega|_{\ti M_0}=\frac{1}{|\bbS^n|}\int_{\bbS^n} sp_kd\omega|_{\ti M_t}\leq \max \ti{s}^{k+1}.
}
By \Cref{pinching est} and \cite[Thm.~5.1]{Andrews1994a} or \cite[Lem.~8.5.8]{Schneider2013a}, the ratio between the maximum width $\ti w_+$ and the minimum width $\ti w_-$ of $\ti M_t$ is uniformly bounded above. Since
\eq{\ti w_+\geq \max \ti s\geq b,}
 we obtain
\eq{\ti w_-\geq c}
for some constant $c.$ Moreover, for any convex body $K$,
\eq{
\tfrac{1}{n+2}w_-(K)B\subset K-C(K)\subset \tfrac{n+1}{n+2}w_+(K)B,
}
see \cite[p.~320]{Schneider2013a}).
Thus the support function of $\ti M_t$ and hence its in-radius is bounded below away from zero independent of time.
}
\begin{lemma}\label{global term-control}
$\mu$ is uniformly bounded above, and below away from zero.
\end{lemma}
\pf{Due to \Cref{pinching est}, for a constant depending on $\varepsilon$, we have
\eq{
c_{\varepsilon}\int_{\bbS^n} p_{k-1}d\omega\leq \int_{\bbS^n}p_k^{\frac{k-1}{k}}d\omega\leq \br{\int_{\bbS^n} p_kd\omega}^{\frac{k-1}{k}}\br{\int_{\bbS^n}d\omega}^{\frac{1}{k}}.
}
Moreover, by \cite[(5.25)]{Schneider2013a} we have
\eq{
\min \ti s^k\leq& \frac{1}{|\bbS^n|}\int_{\bbS^n}p_{k}d\omega\leq \max \ti s^k,\\
\min \ti s^{k-1}\leq&\frac{1}{|\bbS^n|}\int_{\bbS^n}p_{k-1}d\omega.
}
Now the claim follows from \Cref{lower-upper bnds on sp func}.
}

\section{Curvature estimates}\label{Curv-est}
We show that $F$ is uniformly bounded above. In order to eliminate the bad first order term in the evolution of $F$, we will use the first order term in the evolution of the support function. To do this, we need to show that the flow hypersurfaces enclose a ball of small radius for definite amount of time.
We adapt \cite[Lem.~6.1]{McCoy2005} to our needs.

\begin{lemma}\label{starshaped}
If $B_{4\de}(x_{0})\sub K_{t_{0}}$ for some $t_{0}\in [0,T)$, then there exists $c_{\de,\ep}>0$ such that $B_{2\de}(x_{0})\sub K_{t}$ for all $t\in [t_{0},\min(t_{0}+c_{\de,\ep},T))$.
\end{lemma}

\pf{Due to \Cref{pinching est}, there exists $\La_{\varepsilon}$ such that
\eq{\La_{\ep}:=\max_{(0,T)}F^{ij}g_{ij}.}
The solution $r=r(t)$ of
\eq{
\left\{
  \begin{array}{ll}
    \dot{r}&=-\fr{\La_{\ep}}{r}, \\
    r(t_{0})&=4\de
  \end{array}
\right.
}
is given by
\eq{r(t)=\rt{16\de^{2}-2\La_{\ep}(t-t_{0})}}
and we have
\eq{r(t_{0}+\tfrac{6\de^{2}}{\La_{\ep}})=2\de.}
Define
\eq{\psi=\abs{x-x_{0}}^{2}-r^{2}.}
Therefore, we have
\eq{\dot\psi=&2(\mu f-F)\ip{x-x_{0}}{\nu}+2\La_{\ep}\\
\abs{x-x_{0}}^{2}_{;ij}=&2g_{ij}-2h_{ij}\ip{x-x_{0}}{\nu}.}
This implies that
\eq{\cL \psi=2\mu f\ip{x-x_{0}}{\nu}+2\La_{\ep}-2F^{ij}g_{ij}\geq 2\mu f\ip{x-x_{0}}{\nu}.}
The latter term is non-negative, as long as $x_{0}\in K_{t}$. From a continuity argument we conclude that $\min \psi$ is non-decreasing up to $t_{0}+\tfrac{6\de^{2}}{\La_{\ep}}$.
}

\begin{lemma}\label{upper bound on F}
The function $F$ and the Weingarten map $A$ are uniformly bounded along the flow \eqref{gen-flow}.
\end{lemma}

\pf{Let $t_{0}\in [0,T)$ and define
\eq{Z=\frac{F}{\ip{x-x_{0}}{\nu}-\de},}
where $\de>0$ and $x_{0}\in K_{t_{0}}$ are chosen such that
\eq{\ip{x-x_{0}}{\nu}\geq 2\de\q \fa t\in [t_{0},\min(t_{0}+c_{\de,\ep},T));}
see \Cref{lower-upper bnds on sp func} and \Cref{starshaped}. We may assume that $x_{0}=0$. By \eqref{Ev-F}, \eqref{Ev-s} and \eqref{strict-convex}, at a maximum point of $Z$,
\eq{\cL Z\leq& -\fr{\de}{\ip{x}{\nu}-\de} F^{ij}h_{ik}h_{j}^{k}Z-\fr{\mu}{\ip{x}{\nu}-\de}{F_{;}}^{k}\bar\n f(x_{;k})\\
			&+2Z^{2}+\mu\fr{\ip{x}{x_{;i}}}{(\ip{x}{\nu}-\de)^{2}}\bar\n f(x_{;k})h^{ki}F.}
Recall that $\ka_{n}\leq c_{\ep}\ka_1$. This in particular yields
\eq{F^{ij}h_{ik}h^{k}_{j}\geq c_{\ep}F^{2}.}
In addition, due to
\eq{0=Z_{;k}=\fr{F_{;k}}{\ip{x}{\nu}-\de}-\fr{F}{(\ip{x}{\nu}-\de)^{2}}h^{i}_{k}\ip{x}{x_{;i}},}
 we have
\eq{\abs{\n F}\leq c\abs{x}Z^{2}.}
Putting all these together and using \Cref{global term-control} we arrive at
\eq{\label{LZ}\cL Z\leq & -c Z^{3}+2Z^{2}+c\abs{x}Z^{2}.}
Since the origin is in the interior of convex body $K_t,$
\eq{|x|\leq \max |x|=\max s\leq w_+(M_t).}
Thus \Cref{lower-upper bnds on sp func} implies that
\eq{\max_{t\in [t_{0},\min(t_{0}+c_{\de,\ep},T))}Z(t,\cdot)\leq c(M_{0},\max Z(t_{0})).}
A bootstrapping argument starting at $t=0$ shows $F$ is uniformly bounded above. Then $\ka_{n}\leq c_{\ep}F$ yields the bound on $A$.
}

\begin{lemma}
\label{lower bound on F}
There exists a constant $c>0$, which depends on the data of the flow and on $T$, such that along \eqref{gen-flow} we have
\eq{F\geq c.}
In particular, in finite time the principal curvatures range in a compact subset of $\Ga_{+}$.
\end{lemma}

\pf{
By \eqref{Ev-F}, \Cref{pinching est}, and \Cref{global term-control}, at any maximum point of $F^{-1}$ we have
\eq{\cL\br{\fr 1F}\leq \fr{c}{F^{2}}F^{ij}h_{ik}h^{k}_{j}\leq c\fr{\ka_{n}}{F}\leq c,}
where we have used the homogeneity of $F$ and
\eq{\ka_{n}\leq c\ka_{1}\leq cF.}
Therefore, $F^{-1}-ct$ is bounded by its initial value.
}

\section{Long time existence and convergence}\label{sec:LTE}
On a finite time interval, we have collected all the required estimates for establishing the higher order regularity estimates and hence the long-time existence. In this regard, one commonly used procedure is to parameterize the flow hypersurface using a graph representation over a fixed Riemannian manifold. For example, convex hypersurfaces $M$ can be written as a graph in polar coordinates around an interior point $x_0$,
\eq{M=\{(\rho(y),y)\cn y\in \bbS^n\},}
where around $x_0$ we identify the flat Euclidean space with
\eq{(r,y)\in \bbR^{n+1}\bs\{x_0\}=(0,\8)\x\bbS^n,\quad \ip{\cdot}{\cdot}=dr^2+r^2 \bar g(y).}
Along our flow, as long as the graph functions of the flow hypersurfaces are well defined, they satisfy the equation
\eq{\label{Ev-rho}\del_{t}\rho=(\mu f-F)v\equiv \Phi(t,y,\rho,\bar\n\rho,\bar\n^2\rho),}
where $\rho=\rho(t,y)$, cf. \cite[p.~98-99]{Gerhardt:/2006}. In local curvature flow problems, i.e., in absence of a global term, one often uses the regularity theorem of Krylov and Safonov \cite[Sec.~5]{Krylov:/1987} (see e.g., \cite[Thm.~4]{Andrews:/2004} for a clear formulation) to deduce $C^{2,\al}$-estimates for $\rho$ from its $C^{2}$-estimates.

In presence of a global term this practice is unjustified, since the bound one obtains depends on the time derivative of $\Phi$, which is not yet under control in our situation.

Hence we will take a different approach, which has been employed in \cite{Cabezas-Rivas2010a, McCoy2005}, and also in \cite{Makowski:01/2013} where it is presented in great detail. We will make some adjustments to fit our needs.

\begin{lemma}\label{Holder}
Along the flow \eqref{CM-flow} on $[0,T)$, $F=F$ and $\mu f-F$ are H\"older continuous with exponent and H\"older norms independent of $T$. Here both functions are understood to depend on the variable $(t,y)\in [0,T)\x \bbS^n$.
\end{lemma}

\pf{
We transform the evolution equation of $F$ to an equation on the standard round sphere and apply general regularity results. To do this, we will need \Cref{starshaped}.

Note that the inradius is uniformly bounded from below and hence we find $\de>0$, such that for every $t_{0}\in [0,T)$ we have a uniform graph representation of $M_{t}$ over an interior sphere on the time interval
\eq{I_{t_{0}}:= [t_{0},\min(t_{0}+c_{\de,\ep},T)).}
Due to \Cref{starshaped}, the corresponding radial function
\eq{\rho\cn I_{t_{0}}\x \bbS^{n}\ra \bbR}
is uniformly bounded above and below and parameterizes the convex hypersurfaces $M_{t}$. From \cite[Thm.~2.7.10]{Gerhardt:/2006} we infer that
\eq{v^{2}=1+\rho^{-2}\bar g^{ij}\del_{i}\rho\del_{j}\rho}
is uniformly bounded on $I_{t_{0}}$. Since the Weingarten operator is uniformly bounded by \Cref{upper bound on F}, we also have a bound for the second derivatives,
\eq{\abs{\bar\n^{2}\rho}_{\bar g}\leq c(\abs{\bar\n\rho}_{\bar g},\abs{A}).}
Note that none of these bounds depend on $T$.
On $I_{t_{0}}$ and in the above constructed graph coordinates we define:
\eq{a^{ij}=F^{ij},\q b^{m}=F^{kl}(\bar\Ga^{m}_{kl}-\Ga^{m}_{kl})-\mu\bar\n f(x_{;k})g^{km}-(\mu f-F)\nu^m,}
\eq{d_{1}= F^{ij}h_{ik}h^{k}_{j},\q d_{2}=-\mu F^{ij}h_{ik}h^{k}_{j}f-\mu F^{ij}h^{k}_{i}h^{l}_{j}\bar\n^{2}f(x_{;k},x_{;l}),}
where $\bar\Ga$ and $\Ga$ are the Christoffel symbols of the round and the induced metric respectively.
Note that \eqref{Ev-F} is the evolution equation of
\eq{F=F(t,\xi)=F(t,y(x(t,\xi))).}
Since we need the H\"older regularity of $F$ with respect to the $y$-variable as this is the way we deal with $\rho$, we have to adjust the evolution equation \eqref{Ev-F} with the help of the relation
\eq{\partial_tF=\fr{d}{dt}F(t,y(x(t,\xi)))=\fr{\del F}{\del t}(t,y)+\bar\n_m F (t,y)\del_t x^m(t,\xi). }
Hence in the $(t,y)$ coordinates,
\eq{F\cn I_{t_{0}}\x\bbS^{n}\ra \bbR}
 satisfies the equation
\eq{\label{Holder-1}-\del_{t}F+ a^{ij}\bar\n^{2}_{ij}F+b^{m}\bar\n_{m}F+d_{1}F+d_{2}=0.}
Here
\eq{a^{ij}=g^{ik}F^{j}_{k}}
is uniformly elliptic (note that $g_{ij}$ is equivalent to $\bar g_{ij}$) due to the pinching, and $d_{1}$, $\abs{d_{2}}$ are bounded. In order to see that $b^{m}$ is bounded as well, first we note that $\nu^m$ is bounded with respect to $\bar g$ due to the convexity, and then we have to convince ourselves that the tensor $\bar\Ga^{m}_{kl}-\Ga^{m}_{kl}$
 is bounded. However, since
 \eq{g_{ij}=\del_{i}\rho\del_{j}\rho+\rho^{2}\bar g_{ij},}
 we see that $\Ga^{m}_{kl}$ is just an algebraic combination of bounded quantities (a similar argument was used in \cite[Lem.~7.6]{Gerhardt:/2015}). Hence \eqref{Holder-1} has bounded coefficients and neither their bounds nor the ellipticity constants depend on $T$.
Hence we can apply the regularity result \cite[Cor.~7.41]{Lieberman:/1998}, see also \cite[Thm.~3.25]{Schnurer:PDE2}, to obtain the H\"older regularity of $F$.
A very similar argument applies to $\mu f-F$.
}

In the next lemma we prove the $C^{2,\alpha}$ estimates for $\rho$ in several steps.

\begin{lemma}\label{C2alpha}
For any $t_{0}\in [0,T),$ let $c_{\de,\ep}, x_{0}$ be as in \Cref{starshaped} and
\eq{\rho\cn I_{t_{0}}\x\bbS^{n}\ra \bbR}
be a graph representation of $M_{t}$ in polar coordinates around $x_{0}$,
where
\eq{I_{t_{0}}:= [t_{0},\min(t_{0}+c_{\de,\ep},T)).}
Then for some uniform $\al>0$ we have
\enum{
\item \eq{\rho(t,\cdot)\in C^{2,\al}(\bbS^{n})\q\fa t\in I_{t_{0}},}
\item \eq{\bar\n \rho(\cdot,y)\in C^{0,\fr{1+\al}{2}}(I_{t_{0}})\quad 	\mbox{uniformly in}~ y\in \bbS^{n},}
\item \eq{\bar\n^{2}\rho(\cdot,y)\in C^{0,\fr{\al}{2}}(I_{t_{0}})\quad 	\mbox{uniformly in}~ y\in \bbS^{n}.}
}
The bounds on the $C^{2,\al}$-norm depend on the data of the problem and are independent of $T$. In particular we also have
\eq{\abs{\rho}_{2+\al,\fr{2+\al}{2};I_{t_{0}}}\leq \const,}
where the constant only depends on the data of the problem.
\end{lemma}

\pf{
(i)~Fix $t_{1}\in I_{t_{0}}$ and note that due to \Cref{Holder},
\eq{\psi_{t_{1}}:=\mu(t_{1})f(\nu(t_{1},\cdot))-\fr{\del_{t}\rho(t_{1},\cdot)}{v(t_{1},\cdot)}\in C^{0,\al}(\bbS^{n}).}
Equation \eqref{Ev-rho} implies that $\rho(t_{1},\cdot)$ solves the elliptic equation
\eq{\label{spatial-1}F(\cdot,\rho,\bar\n\rho,\bar\n^{2}\rho)=\psi_{t_{1}}}
with a concave operator $F$.
Since we have uniform spatial $C^{2}$-estimates for $\rho$, we can treat $\rho$ and $\bar\n\rho$ as data of the operator and view \eqref{spatial-1} as the problem
\eq{\label{spatial-2}G(\cdot,\bar\n^{2}\rho)=\psi_{t_{1}}.}
After invoking the Bellmann-extension to make $F$ well-defined on $\bbS^n\x\bbR\x\bbR^{n}\x\bbR^{n}_{\mrm{sym}}$ as in \cite[Lem.~7.2]{Makowski:01/2013}, we can apply a result of Caffarelli--Cabr\'e \cite[Thm.~8.1]{CaffarelliCabre:/1995} and the subsequent remarks (see also \cite[Thm.~7.3]{Makowski:01/2013}) to \eqref{spatial-2} to obtain the spatial H\"older regularity at time $t_{1}$. Since none of the constants involved here depend on $T$ and $t_{1}$ is arbitrary, the result follows.

To prove (ii) and (iii), we apply Andrews' estimates from \cite[Sec.~3.3, 3.4]{Andrews:/2004}.
To do this, we need to deduce an equation for the difference quotient in any direction $e_{\ell}$, $1\leq \ell\leq n$,
\eq{\de_{\tau}\rho(t,\xi)=\fr{\rho(t,\xi+\tau e_{\ell})-\rho(t,\xi)}{\tau}.}
Here we are considering $\xi\in \bbR^{n}$ to be in a fixed coordinate chart of $\bbS^{n}$. A similar calculation appeared in \cite[Lem.~4.1]{Sani:/2017}. Recalling that
\eq{\del_{t}\rho(t,\cdot)&=(\mu(t)f(\cdot,\rho,\bar\n\rho)-F(\cdot,\rho,\bar\n\rho,\bar\n^{2}\rho))v\\
				&=(\mu(t)f(\nu)-F(A))v,}
we calculate the evolution of $\de_{\tau}\rho$ as follows:
\eq{\del_{t}\de_{\tau}\rho=\tau^{-1}((\mu f-F(A))_{|(t,\xi+\tau e_{\ell})}-(\mu f-F(A))_{|(t,\xi)})v.}
The crucial term arises from $F$. Hence, using that the domain of $F$ is convex, we continue with
\eq{&\tau^{-1}(vF(A)_{|(t,\xi+\tau e_{\ell})}-vF(A)_{|(t,\xi)})\\
	=&\tau^{-1}\int_{0}^{1}\fr{d}{ds}\br{v(t,\xi+s\tau e_{\ell})F\br{s A_{|(t,\xi+\tau e_{\ell})}+(1-s)A_{|(t,\xi)}}}~ds\\
	=&\int_{0}^{1}\fr{\del v}{\del \xi^{\ell}}(s) F\br{s A_{|(t,\xi+\tau e_{\ell})}+(1-s) A_{|(t,\xi)}}~ds\\
	&+\tau^{-1}\int_{0}^{1}v(s)F^{j}_{i}(s)(h^{i}_{j}(t,\xi+\tau e_{\ell})-h^{i}_{j}(t,\xi))~ds}
We can use \cite[(1.5.10), Lem.~2.7.6]{Gerhardt:/2006} write $h$ in terms of $\bar\n^{2}\rho$,
\eq{h^{i}_{j}=\fr{1}{v\rho}\de^{i}_{j}+\fr{1}{v^{3}\rho^{3}}\bar \n^{i}\rho\bar\n_{j}\rho-\fr{\hat g^{ik}}{v\rho^{2}}\bar\n^{2}_{kj}\rho,}
where $\hat g^{ik}$ is the inverse of
\eq{\hat g_{ik}=\rho^{-2}\bar\n_{i} \rho\bar\n_{k} \rho+\bar g_{ik}.}
Defining
\eq{\cF^{j}_{i}:=\int_{0}^{1}v(t,\xi+s\tau e_{\ell})F^{j}_i(s A_{|(t,\xi+\tau e_{\ell})}+(1-s) A_{|(t,\xi)})~ds}
and using that the second spatial derivatives as well as $\mu$ are bounded, we see that $\de_{\tau}\rho$ satisfies the equation
\eq{\del_{t}\de_{\tau}\rho&=-\cF^{j}_{i}\de_{\tau}h^{i}_{j}+\Psi_1=\fr{\hat g^{ik}}{v\rho^{2}}\cF^{j}_{i}\bar\n^{2}_{kj}\de_{\tau}\rho+\Psi_2,}
 and where uniformly in $t$ there holds
\eq{\abs{\Psi_2(t,\cdot)}\leq c(\ep,\abs{\rho(t,\cdot)}_{C^{2}(\bbS^{n})}).}
Note that $\cF^{j}_{i}$ is uniformly elliptic, since $A_{|(t,\xi+\tau e_{\ell})}$ as well as $A_{|(t,\xi)}$ both range in a closed, convex and strict subcone of $\Ga_{+}$, due to the pinching.
Now the estimates from \cite[Sec.~3.3, 3.4]{Andrews:/2004} go through and yield (ii) and (iii). The full parabolic H\"older estimates (cf., \cite[Ex.~2.5.3]{Gerhardt:/2006}) now follow from the definition of $\mu$, which gives
\eq{\mu\in C^{\fr{\al}{2}}([0,T)),}
and finishes the proof of \Cref{C2alpha}.
}

A standard bootstrapping argument involving Schauder estimates as in \cite[Thm.~2.5.9]{Gerhardt:/2006} yields the following proposition.

\begin{prop}\label{LTE}
For an arbitrary strictly convex initial hypersurface the flow \eqref{gen-flow} exists for all times.
\end{prop}

Due to the lack of a uniform lower bound on $F$, we cannot yet obtain \textit{uniform} $C^{\8}$-estimates on $[0,\8)$, as the principal curvatures might leave every compact subset of $\Ga_{+}$. In the next lemma, the missing uniform lower $F$ bound will be obtained by combining a local maximum principle \cite[Thm.~7.36]{Lieberman:/1998} with the weak Harnack inequality \cite[Thm.~7.37]{Lieberman:/1998} (see also the proof of \cite[Lem.~7.7]{Gerhardt:/2015}). Together with the pinching estimate, this ensures uniform estimates on the principal curvatures and allows us to perform the above procedure with $T=\infty$ to obtain the uniform $C^{\8}$-estimates.

\begin{lemma}\label{Harnack}
Along the flow \eqref{gen-flow}, $F$ is uniformly bounded from below,
\eq{F\geq \const >0.}
In particular, during the whole evolution the principal curvatures range in a compact subset of $\Ga_{+}$.
\end{lemma}

\pf{
Returning to the proof of \Cref{Holder}, we see from \eqref{Holder-1} that
\eq{L_{0}F:=-\del_{t}F+a^{ij}\bar\n_{ij}^{2}F+b^{m}\bar\n_{m}F=-d_{1}F-d_{2}}
in coordinate systems which are valid on intervals of fixed uniform length $c_{\de,\ep}$. We have uniform bounds on the coefficients and on the ellipticity constants. Due to the pinching estimate, \Cref{global term-control} and \Cref{upper bound on F} there holds
\eq{-\ga F\leq L_{0}F\leq \ga F }
for some constant $\ga>0$. For $0<\al\leq 1$ we define
\eq{I_{\al}=[t_{0}+(1-\al)c_{\de,\ep},t_{0}+c_{\de,\ep}).}
Applying \cite[Thm.~7.36]{Lieberman:/1998} to
\eq{L_{0}(e^{-\ga(t-t_0)}F)\geq 0,} we obtain for any $p>0$,
\eq{\label{Harnack-1}\sup_{I_{\al_{1}}\x\bbS^{n}}F\leq e^{\ga c_{\de,\ep}}\sup_{I_{\al_{1}}\x\bbS^{n}}e^{-\ga(t-t_0)}F\leq c\br{\int_{I_{\al_{2}}\x\bbS^{n}}F^{p}}^{\fr 1p},}
for sufficiently small $\al_{1}<\al_{2}$,
where $c$ depends on our previous established bounds and on $p$. Here we have implicitly used a standard covering argument to get from local parabolic cylinders to the whole of $\bbS^{n}.$ Note that for every time $t\in [0,\8)$, we have
\eq{\sup_{M_{t}}F\geq c_{1}>0}
for a small constant. This can be seen from comparison with a circumscribed sphere, the radius of which is under control due to the support function bound.

Now we define
\eq{\hat F=e^{\ga(t-t_{0})}F}
and obtain that
\eq{L_{0}\hat F=e^{\ga(t-t_{0})}(L_{0} F-\ga F)\leq 0.}
We apply \cite[Thm.~7.37]{Lieberman:/1998} and find $p>0$, such that
\eq{0<c_{1}\leq\br{\int_{[t_{0},t_{0}+\al_{3})\x\bbS^{n}}\hat F^{p}}^{\fr 1p}\leq c\inf_{I_{\al_{4}}}\hat F=c\inf_{I_{\al_{4}}}e^{\ga(t-t_{0})}F\leq c\inf_{I_{\al_{4}}}F, }
where $\al_{3}$ and $\al_{4}$ are sufficiently small and the first bound follows from \eqref{Harnack-1} and the arbitrariness of $t_{0}$. Note that \cite[Thm.~7.37]{Lieberman:/1998} requires a waiting time of $3R^{2}$, where $R$ is the size of the parabolic cylinder that is used. We can easily set up a uniform $R$ depending on the lower bound on $c_{\de,\ep}$ and a covering of $\bbS^{n}$ by open balls of fixed size. This completes the proof, since $t_{0}$ is arbitrary.
}

With this lower bound on $F$, we can obtain time-independent regularity estimates with a bootstrapping argument:

\begin{lemma}\label{full-reg-est}
Along the flow \eqref{gen-flow}, the hypersurfaces
\eq{\ti M_{t}=M_{t}-C(K_{t})}
enjoy uniform $C^{\8}$-estimates.
\end{lemma}

We complete the proof of \Cref{CM-flow}. By \eqref{derivative of functional}, we have
\eq{
\fr{d}{dt}\int_{\bbS^{n}}\tilde{s}\varphi d\om
=&\fr{\int_{\bbS^{n}}\tilde{p}_{k}^{\fr{k-1}{k}}d\om}{\int_{\bbS^{n}}\ph^{-\fr 1k} \tilde{p}_{k}d\om}\int_{\bbS^{n}}\ph^{\fr{k-1}{k}}d\omega-\int_{\mathbb{S}^n}\ph \tilde{p}_{k}^{-\fr 1k}d\omega\leq0 .
}
Since $\int_{\bbS^{n}}\tilde{s}\varphi d\om$ is bounded along the flow, \Cref{full-reg-est} implies that $\tilde{M}_t$ subsequentially converges to solutions of
\eq{\label{limit-equation}p_k=\gamma\varphi,~\mbox{where}~\gamma:=\frac{\int_{\bbS^{n}}sp_kd\om}{\int_{\bbS^{n}}s\varphi d\om}.}
Due to \Cref{Monotone Q}, all limits of $\tilde{M}_t$ share the same value for $\gamma$.
By the Alexandrov-Fenchel-Jessen Theorem, the solution of \eqref{limit-equation} is unique up to translations. Since any limit of $\tilde{M}_t$ has its centroid at the origin, the convergence is in fact independent of subsequences.

\section{The weakly convex case and completion of the proof}\label{completion}
In this section, we use a simple approximation together with the constant rank theorem and \Cref{CM-flow} to solve the Christoffel--Minkowski problem under the weaker convexity assumption on $\varphi,$ i.e.,
\eq{\bar \n^{2}\varphi^{-\frac{1}{k}}+\varphi^{-\frac{1}{k}}\bar g\geq  0,\quad \int_{\mathbb{S}^n}u\varphi(u)d\om=0.}
We mention that an approximation argument appeared \cite[Sec. 4]{Sheng2004} to avoid the homotopy assumption \eqref{homotopy} used in \cite{Guan2003a}; however, their argument is more complicated.
The following lemma gives a simpler proof for the validity of this homotopy assumption. We even prove that the homotopy can be chosen to be strictly convex for all $\tau<1$.

\begin{lemma}\label{Approx}
Suppose $1\leq k\leq n $ and $0<\varphi\in C^{\infty}(\mathbb{S}^n)$ satisfies
\eq{\bar \n^{2}\varphi^{-\frac{1}{k}}+\varphi^{-\frac{1}{k}}\bar g\geq  0,\quad \int_{\mathbb{S}^n}u\varphi(u)d\om=0.}
Then for each $\tau\in [0,1)$, there exists $z_{\tau}\in \bbR^{n+1}$, such that
\eq{\varphi_{\tau}(u):=(1-\tau+\tau\varphi^{-\frac{1}{k}}(u)-\langle u,z_{\tau}\rangle)^{-k}}
satisfies
\eq{\bar \n^{2}\varphi_{\tau}^{-\frac{1}{k}}+\varphi_{\tau}^{-\frac{1}{k}}\bar g>  0,\quad \int_{\mathbb{S}^n}u\varphi_{\tau}(u)d\om=0.}
Moreover, we have
\eq{\label{Approx-1}|z_{\tau}|\leq 1+\max\varphi^{-\frac{1}{k}}.}
\end{lemma}
\begin{proof}
Note that for $0\leq\tau\leq 1$,
\eq{s_{L_{\tau}}:=1-\tau+\tau\varphi^{-\frac{1}{k}}} is the support function of a convex body $L_{\tau}$, which is smooth and strictly convex for $\tau<1$. By \cite[Lem. 3.1]{Ivaki2016b}, there exists a unique point $z_{\tau}$ in the interior of $L_{\tau}$ such that
\eq{\begin{cases}
        \min\limits_{v\in L_{\tau}}\int_{\mathbb{S}^n}-\log(s_{L_{\tau}-v})d\omega, & k=1 \\
        \min\limits_{v\in L_{\tau}}\int_{\mathbb{S}^n}s_{L_{\tau}-v}^{-k+1}d\omega, & k>1
      \end{cases}
  }
is attained. Hence the support function of $L_{\tau}-z_{\tau}$ given by $\varphi_{\tau}^{-\frac{1}{k}}$ is positive and satisfies the required integral condition, see \cite[Lem.~3.1]{Ivaki2016b}. Since $z_{\tau}$ is in the interior of $L_{\tau}$, the upper bound on the norm of $z_{\tau}$ follows.
\end{proof}

We complete the proof of \Cref{CMP}.

\begin{proof}[Proof of \Cref{CMP}]
Consider the family of functions $\varphi_{\tau}$ in the previous lemma. Due to \Cref{CM-flow}, for each $\tau\neq 1$, there exists a smooth, strictly convex hypersurface with support function $s_{\tau}$ that solves
\eq{p_k(\bar \n^{2}s_{\tau}+s_{\tau}\bar g )=\varphi_{\tau}.}
Since $L_{\tau}$ converges to $L_1$ in the Hausdorff distance as $\tau \to 1$, by \cite[Thm 7.1]{Ivaki2016b} we have $\lim_{\tau\to 1}z_{\tau}=0.$ This in turn implies that $\varphi_{\tau}$ converges smoothly to $\varphi$.

Now due to \cite[Thm. 3.3]{Guan2003a}, we have uniform $C^{k}$ estimates on $s_{\tau}$, independent of $\tau.$
%As $\si_k$ is a smooth map, we also have uniform $C^k$-estimates for the functions $\varphi_{\tau}$.
Therefore, a subsequence of $s_{\tau}$ converges to a smooth function $s$ such that $\bar \n^{2}s+s\bar g \geq 0$ and
\eq{p_k(\bar \n^{2}s+s\bar g )=\varphi.}
By the constant rank theorem \cite[Thm.~1.2]{Guan2003a}, we have $\bar \n^{2}s+s\bar g> 0$ and the corresponding hypersurface is strictly convex.
\end{proof}

\section*{Acknowledgment}
PB was supported by the ARC within the research grant ``Analysis of fully non-linear geometric problems and differential equations", number DE180100110. MI was supported by a Jerrold E. Marsden postdoctoral fellowship from the Fields Institute. JS was supported by the ``Deutsche Forschungsgemeinschaft" (DFG, German research foundation) within the research scholarship ``Quermassintegral preserving local curvature flows", grant number SCHE 1879/3-1.

\bibliographystyle{amsalpha-nobysame}
\bibliography{library,Bibliography}
\vspace{10mm}

\textsc{Department of Mathematics, Macquarie University,\\ NSW 2109, Australia, }\email{\href{mailto:paul.bryan@mq.edu.au}{paul.bryan@mq.edu.au}}

\vspace{2mm}

\textsc{Department of Mathematics, University of Toronto,\\ Ontario, M5S 2E4, Canada, }\email{\href{mailto:m.ivaki@utoronto.ca}{m.ivaki@utoronto.ca}}

\vspace{2mm}

\textsc{Department of Mathematics, Columbia University,\\ New York, NY 10027, USA,} \email{\href{mailto:julian.scheuer@math.uni-freiburg.de}{julian.scheuer@math.uni-freiburg.de}}
\end{document}